\documentclass[12pt]{umj}
\usepackage[T1,T2A]{fontenc}
\usepackage[cp1251]{inputenc}
\usepackage[english,russian]{babel}
\usepackage{amsmath}
\usepackage{amssymb}
\usepackage{amsfonts, dsfont, graphicx, srcltx}
\usepackage{cite}

\newtheorem{lemma}{Лемма} 
\newtheorem{theorem}{Теорема} 

\theoremstyle{definition} 
\newtheorem{example}{Пример}

\def\udcs{517.9} 

\firstpage{1}

\begin{document}

УДК \udcs
 \thispagestyle{empty}

\title{Об успокоении системы управления на графе-звезде с глобальным запаздыванием, пропорциональным времени} 

\author{А.П. Леднов}
\address{Александр Петрович Леднов, 
\newline\hphantom{iii} Саратовский национальный исследовательский государственный университет имени Н.Г. Чернышевского,
\newline\hphantom{iii} ул. Астраханская, 83, 
\newline\hphantom{iii} 410012, г. Саратов, Россия
\smallskip
\newline\hphantom{iii} Московский центр фундаментальной и прикладной математики,
\smallskip
\newline\hphantom{iii} Московский государственный университет имени М.В. Ломоносова,
\newline\hphantom{iii} ул. Ленинские Горы, 1,
\newline\hphantom{iii} 119991, г. Москва, Россия
}
\email{lednovalexsandr@gmail.com}

\thanks{\sc Lednov A.P. 
On damping a control system on a star graph with global time--proportional delay}
\thanks{\copyright \ Леднов А.П. 2024}
\thanks{\rm Исследование выполнено за счет гранта Российского научного фонда № 24-71-10003, https://rscf.ru/project/24-71-10003/, в МГУ им. М. В. Ломоносова.}

\maketitle 
{
\small
\begin{quote}
\noindent{\bf Аннотация. } 
Рассматривается задача об успокоении системы управления с запаздыванием, описываемой функционально-дифференциальными уравнениями первого порядка на временн\'ом графе-звезде. Запаздывание в системе пропорционально времени и распространяется через внутреннюю вершину. Исследуется вариационная задача о минимизации функционала энергии с учетом вероятностей сценариев, соответствующих различным ребрам. Установлено, что оптимальная траектория удовлетворяет условиям типа Кирхгофа во внутренней вершине. Доказана эквивалентность вариационной задачи некоторой краевой задаче для функционально-дифференциальных уравнений второго порядка на графе, и установлена однозначная разрешимость обеих задач.
\medskip

\noindent{\bf Ключевые слова:}{ уравнение пантографа, функционально-дифференциальное уравнение, квантовый граф, глобальное запаздывание, временн\'ой граф, вариационная задача, оптимальное управление.}

\medskip

\begin{center}
    \bfseries\normalsize On damping a control system on a star graph with global time-proportional delay
\end{center}

\noindent{\bf Abstract.}  We consider the problem of damping a control system with delay, described by first-order functional-differential equations on a temporal star graph. The delay in the system is time-proportional and propagates through the internal vertex. We study the variational problem of minimizing the energy functional, taking into account the probabilities the of scenarios corresponding to different edges. It is established that the optimal trajectory satisfies Kirchhoff-type conditions at the internal vertex. The equivalence of the variational problem to a certain boundary value problem for second-order functional-differential equations on the graph is proved, and the unique solvability of both problems is established.

\medskip


\medskip

 
\end{quote}
}

\section{Введение}

Н.Н. Красовским \cite{1} была поставлена и исследована задача об успокоении управляемой системы с последействием. Рассматривалась система с постоянным запаздыванием, описываемая уравнением запаздывающего типа. Более сложный случай, когда уравнение имеет нейтральный тип, т.е. содержит запаздывание и в главных членах, был рассмотрен А.Л. Скубачевским \cite{2,14}  и позже другими авторами (см. также \cite{-12} и литературу там). Недавно С.А. Бутерин перенес данную задачу на так называемые временные графы \cite{3,4,12,13}.

Дифференциальные операторы на графах, часто называемые квантовыми графами, активно изучаются с прошлого века в связи с моделированием различных процессов, протекающих в сложных системах, представимых в виде пространственных сетей \cite{7,8,9,10,-10}. Для таких моделей характерны условия типа Кирхгофа во внутренних вершинах. В \cite{3,4,12,13} показано, что аналогичным условиям  удовлетворяют оптимальные траектории на временн\'ых графах. 

В \cite{5} было предложено определение операторов на графах с глобальным запаздыванием. Последнее означает, что запаздывание распространяется через внутренние вершины графа. Другими словами,  решение уравнения на входящем ребре служит начальной функцией для уравнений на исходящих ребрах. Глобальное запаздывание стало альтернативой локально нелокальному случаю рассмотренному в \cite{11}, когда уравнение на каждом ребре имеет свой собственный параметр запаздывания и может быть решено отдельно от уравнений на остальных ребрах.

Использование концепции глобального запаздывания позволило перенести на графы вышеупомянутую задачу об успокоении системы управления с последействием в \cite{3,4}. Это, в свою очередь, привело к концепции временн\'ого графа, ребра которого, в отличие от пространственной сети, отождествляются с промежутками времени, а каждая внутренняя вершина понимается как точка разветвления процесса, дающая несколько различных сценариев дальнейшего его протекания. Кроме того, в \cite{3,12} была предложена стохастическая интерпретация системы управления на временн\'ом дереве. А именно, к системе на дереве приведет, например, замена коэффициентов в уравнении на интервале дискретными случайными процессами с дискретным временем. В работе \cite{13} исследовалась глобально нелокальная интегро-дифференциальная система управления на графе-звезде.

В настоящей работе на временн\'ой граф типа звезды распространяется задача об успокоении системы управления, описываемой так называемым уравнением пантографа, которое имеет ряд важных приложений \cite{15,16,17,18,19,20,6}. В данном случае запаздывание не постоянно, а является пропорциональным времени сжатием.

На интервале этот случай был рассмотрен Л.Е. Россовским в \cite{6} для уравнения нейтрального типа:
\begin{equation}
y^{\prime}\left(t\right)+a y^{\prime}\left(q^{-1} t\right)+b y\left(t\right)+c y\left(q^{-1} t\right)=u\left(t\right), \quad t>0, \label{1}
\end{equation}
где $a, b, c \in \mathbb{R}$ и $q>1$ , а $u\left(t\right)$ -- управляющее воздействие, которое является вещественнозначной функцией. Состояние системы в начальный момент времени задается условием
\begin{equation}y\left(0\right)=y_0 \in \mathbb{R}. \label{2}\end{equation}
Задача управления формулируется следующим образом. Требуется найти $u(t)\in L_2(0,T)$, приводящее систему (\ref{1}), (\ref{2}) в равновесие $y(t)=0$ при $t \geqslant T$ для некоторого $T>0$.

Для этого достаточно найти $u(t)$, приводящее систему в состояние 
\begin{equation}y\left(t\right)=0, \quad q^{-1}T \leqslant t \leqslant T, \label{3}\end{equation}
а затем сбросить управление, положив $u\left(t\right) \equiv 0$ при $t>T$.  При этом из всех возможных управлений ищется управление, обладающее минимальной энергией
$$ \displaystyle\int_0^{T} u^2\left(t\right) d t.$$

В результате получается задача минимизации квадратичного функционала
\begin{equation}\mathcal{J}\left(y\right)=\displaystyle\int_0^{T}\left(y^{\prime}\left(t\right)+a y^{\prime}\left(q^{-1} t\right) +b y\left(t\right)+c y\left(q^{-1} t\right)\right)^2 d t \longrightarrow \min \label{4}\end{equation}
на множестве функций $y\left(t\right)\in W^1_2[0,T]$, удовлетворяющих краевым условиям  (\ref{2}), (\ref{3}). 

Решение задачи (\ref{2})--(\ref{4}) получено в \cite{6}. Оно сведено к решению эквивалентной краевой задачи для функционально--дифференциального уравнения второго порядка. В частности, установлено, что если функция  $y\left(t\right)\in W^1_2[0,T]$  удовлетворяет (\ref{2}), (\ref{3}) и минимизирует функционал (\ref{4}), то для всех $v\in W_2^1[0,T]$, удовлетворяющих (\ref{2}) при $y_0=0$ и (\ref{3}), выполняется интегральное тождество
\begin{multline*}
\int_0^{q^{-1}T}\left(\left(1+a^2 q\right) y^{\prime}(t)+a y^{\prime}\left(q^{-1} t\right)+a q y^{\prime}(q t)\right) v^{\prime}(t) d t+ \\
+\int_0^{q^{-1}T}\left(\left(a b-c q^{-1}\right) y^{\prime}\left(q^{-1} t\right)+\left(c q-a b q^2\right) y^{\prime}(q t)+\right.\\
\left.+\left(b^2+c^2 q\right) y(t)+b c y\left(q^{-1} t\right)+b c q y(q t)\right) v(t) d t=0.
\end{multline*}
Последнее означает, что $y\left(t\right)\in W^1_2[0,T]$ определяет обобщенное решение краевой задачи 
\begin{multline}
-\left(\left(1+a^2 q\right) y^{\prime}(t)+a y^{\prime}\left(q^{-1} t\right)+a q y^{\prime}(q t)\right)^{\prime}+\left(a b-c q^{-1}\right) y^{\prime}\left(q^{-1} t\right)+ \\
+\left(c q-a b q^2\right) y^{\prime}(q t)+\left(b^2+c^2 q\right) y(t)+b c y\left(q^{-1} t\right)+b c q y(q t)=0, \quad 0<t<q^{-1}T, 
\label{-1}
\end{multline}
при условиях (\ref{2}) и (\ref{3}). Обратное также верно: если $y\left(t\right)\in W^1_2[0,T]$ является обобщенным решением задачи (\ref{2}), (\ref{3}), (\ref{-1}), то $y$ доставляет минимум функционалу (\ref{4}).

Следующая теорема устанавливает существование и единственность обобщенного решения краевой задачи (\ref{2}), (\ref{3}), (\ref{-1}), а также однозначную разрешимость вариационной задачи (\ref{2})--(\ref{4}).

\begin{theorem}[\cite{6}]\label{Th-2} Пусть $|a| \neq q^{-1 / 2}$. Тогда задача (\ref{2}), (\ref{3}) (\ref{-1}), имеет единственное обобщенное решение $y \in W_2^1[0,T]$. \end{theorem}

В следующем разделе дается постановка вариационной задачи на графе-звезде. В частности, это потребовало понимание того, как должно выглядеть на графе глобальное сжатие. Для простоты мы ограничиваемся случаем $a=0$, т.е. уравнением запаздывающего типа. В третьем разделе будет установлена эквивалентность вариационной задачи некоторой краевой задаче для функционально-дифференциальных уравнений второго порядка на графе. В последнем разделе доказывается однозначная разрешимость обеих задач.

\section{Постановка вариационной задачи на графе звезде} 

Рассмотрим граф $\Gamma_m$, изображенный на рисунке \ref{fig1}. Как обычно, под функцией $y$ на графе $\Gamma_m$ будем понимать кортеж $y=[y_1,\dots,y_m]$, в котором компонента $y_j$ определена на ребре $e_j$, т.е. $y_j=y_j\left(t\right)$,  $t\in[0,T_j]$.

\begin{figure}
\begin{center}
\includegraphics[scale=0.3]{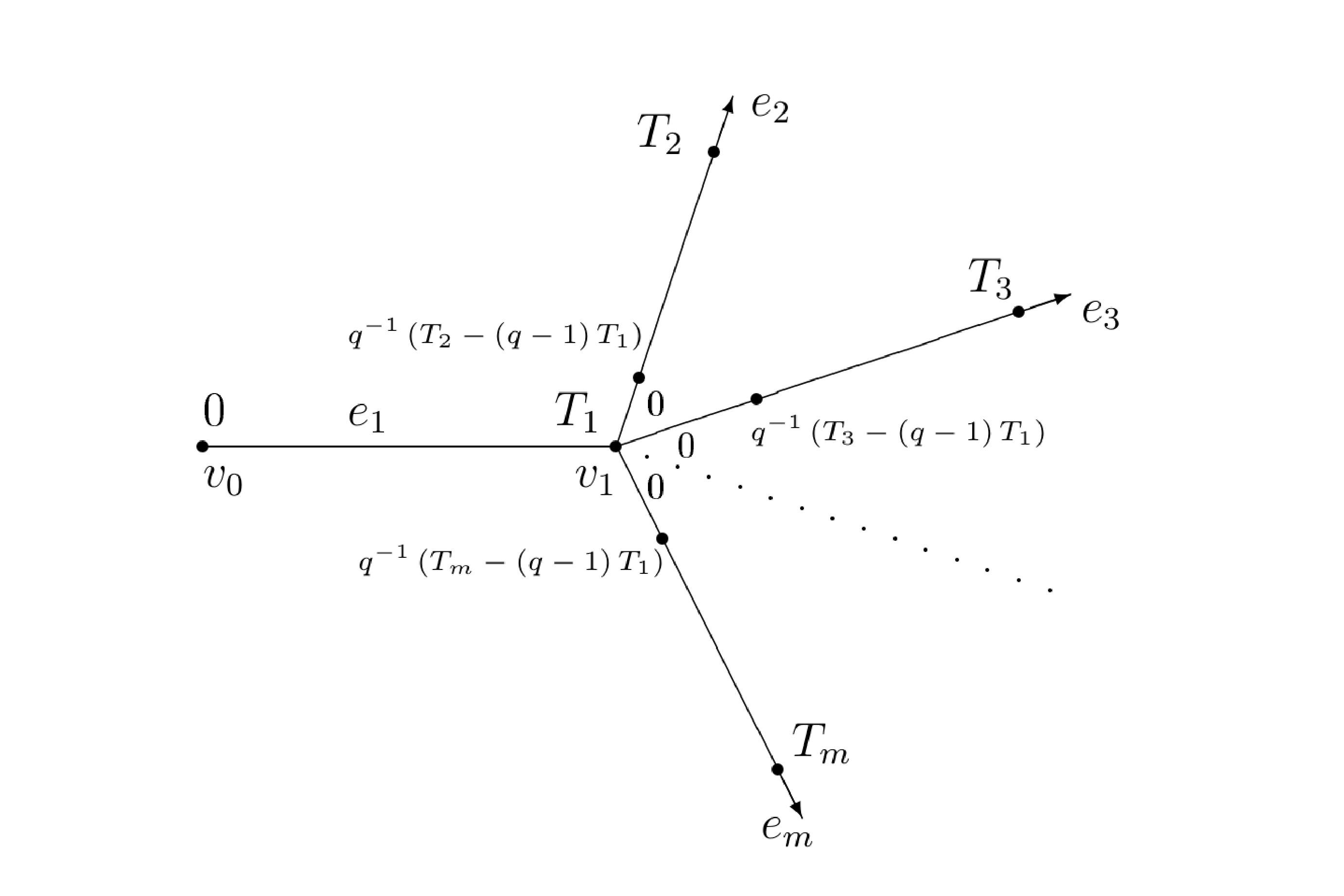}
\end{center}
\caption{\small{Граф типа звезды $\Gamma_m$}}
\label{fig1}
\end{figure}

Пусть до момента времени $t=T_1>0,$ ассоциированного с единственной внутренней вершиной $v_1$ графа $\Gamma_m$, наша система управления с запаздыванием пропорциональным времени на $\Gamma_m$ описывается уравнением
\begin{equation}
\ell_1 y\left(t\right):= y_1^\prime\left(t\right) + b_1y_1\left(t\right) + c_1y_1\left(q^{-1}t\right) = u_1\left(t\right), \quad 0<t<T_1,
\label{11}
\end{equation}
заданным на ребре $e_1$ графа $\Gamma_m$, с начальным условием 
\begin{equation}
y_1\left(0\right)=y_0.
\label{12}
\end{equation}
При $t=T_1$, т.е. в вершине $v_1$, система разветвляется на $m-1$ независимых параллельных процессов, описываемых уравнениями 
\begin{equation}
\ell_j y\left(t\right):= y_j^\prime\left(t\right) + b_jy_j\left(t\right) + c_jy_j\left(q^{-1}\left(t - \left(q-1\right)T_1\right)\right) =u_j\left(t\right),\quad t>0,\quad  j=\overline{2,m}, 
\label{13}
\end{equation}
но имеющих общую историю, определяемую уравнением (\ref{11}) с начальным условием (\ref{12}) и условиями
\begin{equation}
y_j\left(t\right)=y_1\left(t+T_1\right), \quad \left(q^{-1}-1\right)T_1 < t < 0, \quad j=\overline{2,m},
\label{14}
\end{equation}
а также условиями непрерывности в вершине $v_1$, которые в данном случае согласуются с (\ref{14}) при $t\to-0$: 
\begin{equation}
y_j\left(0\right)=y_1\left(T_1\right), \quad  j=\overline{2,m}.
\label{15}
\end{equation}
Как и в предыдущем разделе, мы предполагаем, что $q>1$, $y_0\in\mathbb{R}$ и все $b_j$, $c_j\in\mathbb{R}$.

В (\ref{13}) $j$-ое уравнение задано на ребре $e_j$ графа $\Gamma_m$, представляющем собой, вообще говоря бесконечный луч, выходящий из вершины $v_1$. Условия (\ref{14}) означают, что запаздывание распространяется через вершину $v_1$.

Задача Коши (\ref{11})–-(\ref{15}) имеет единственное решение $y_j(t) \in W_2^1[0, T_j]$, $j = \overline{1,m}$, для любого фиксированного $T_j > 0 $, $j = \overline{2,m}$, при условии, что $u_j(t) \in L_2(0, T_j)$ для $j = \overline{1,m}$. 
На ребре $e_1$ она сводится к уравнению Вольтерра  второго рода с кусочно--постоянным ядром и свободным членом из $W_2^1[0,T_1]$:
$$y_1(t)=y_0+\displaystyle\int_0^{t} u_1(s)ds-b_1\int_0^t y_1(s) ds -qc_1\displaystyle\int_0^{q^{-1}t}y_1(s) ds=f(t)+\displaystyle\int_0^{t}K(t,s)y_1(s)ds,$$
где $$f(t)=y_0+\displaystyle\int_0^{t} u_1(s)ds,$$
$$K(t, s) = 
\begin{cases}
- (b_1 +  qc_1), &  s \leq q^{-1}t, \\
- b_1, & s > q^{-1}t.
\end{cases}$$ 
Тогда полученная функция $y_1(t)$ дает общую начальную функцию (\ref{14}) и общее начальное условие (\ref{15}) для всех уравнений в (\ref{13}), которые могут быть решены аналогично.

\begin{example}
Пусть $m = 2$, $b := b_1 = b_2$, $c := c_1 = c_2$, и  
$$
y(t) := 
\begin{cases} 
y_1(t), & 0 \leq t \leq T_1, \\
y_2(t - T_1), & t > T_1, 
\end{cases}
\quad
u(t) := 
\begin{cases} 
u_1(t), & 0 < t < T_1, \\
u_2(t - T_1), & t > T_1.
\end{cases}
$$
Тогда задача Коши (\ref{11})–-(\ref{15}) принимает вид (\ref{1}), (\ref{2}) при $a = 0$.
\end{example}

Предположим для определенности, что $T_j>\left(q-1\right)T_1$ при всех $j=\overline{2,m}$. Для успокоения системы (\ref{11})–-(\ref{15}) сразу при всех сценариях  нужно привести ее в состояние
\begin{equation}
y_j\left(t\right)=0, \quad q^{-1}\left(T_j-\left(q-1\right)T_1\right)\leqslant t\leqslant T_j, \quad j=\overline{2,m},
\label{16}
\end{equation}
выбрав подходящие управления  $u_j\left(t\right)$, $j=\overline{1,m}$. Тогда, положив $u_j\left(t\right)\equiv0$ при $t>T_j$, $j=\overline{2,m}$, будем иметь $y_j(t)=0$ при тех же $t$ и $j$. Другими словами, система будет приведена в равновесие на каждом исходящем ребре. Поскольку такие $u_j\left(t\right)$ не единственны, будем искать их, минимизируя усилия $\|u_j\|^2_{L_2\left(0,T_j\right)}$.  Кроме того, аналогично тому, как это было сделано в \cite{3} для случая постоянного запаздывания, мы можем регулировать степень участия каждого $\|u_j\|^2_{L_2\left(0,T_j\right)}$ в соответствующем функционале энергии, выбирая определенный положительный вес $\alpha_j$.

Таким образом, приходим к вариационной задаче 
\begin{equation}
\mathcal{J}\left(y\right)=\sum_{j=1}^m\alpha_j\displaystyle\int_0^{T_j}\left(\ell_j y\left(t\right)\right)^2\,dt\to\min
\label{17}
\end{equation}
при условиях (\ref{12}), (\ref{14})--(\ref{16}), где $\alpha_j>0$, $j=\overline{1,m}$, фиксированы.

Для выбора весов $\alpha_j$, $j=\overline{1,m}$, можно применить вероятностный подход в соответствии с  интерпретацией системы управления на временн\'ом графе, предложенной в \cite{3}. А именно, для этого нужно положить $\alpha_1=1$, а в качестве $\alpha_j$, $j=\overline{2,m}$, взять вероятности сценариев, задаваемых соответствующими уравнениями в (\ref{13}), т.е. $\alpha_2+\ldots+\alpha_m=1$. Последнее тождество также обеспечивает соответствие случаю интервала, если уравнения в (\ref{13}) не зависят от $j$, т.е. являются искусственными копиями одного и того же единственно возможного сценария (см. пример 2 в \cite{3}).

Заметим, что условия (\ref{14}) никаких ограничений на функцию $y=[y_1,\dots,y_m]$ не накладывают. Поэтому условимся, что взятие $\mathcal{J}\left(y\right)$, равно как и $\ell_j y$ при $j=\overline{2,m}$, от какой бы то ни было функции $y$ на $\Gamma_m$ автоматически подразумевает применение условий (\ref{14}). Для краткости также введем обозначение $\ell y:=\left[\ell_1y,\dots,\ell_m\right]$.

\section{Сведение к краевой задаче}

Рассмотрим вещественное гильбертово пространство $W_2^k\left(\Gamma_m\right)=\bigoplus_{j=1}^m W_2^k[0,T_j]$ со скалярным произведением $$\left(y,z\right)_{W_2^k\left(\Gamma_m\right)}=\sum_{j=1}^m\left(y_j,z_j\right)_{W_2^k[0,T_j]},$$ где $y=[y_1,\dots,y_m]$, $z=[z_1,\dots,z_m]$, $\left(f,g\right)_{W_2^k[a,b]}=\sum_{\nu=0}^{k} \left(f^{\left(\nu\right)},g^{\left(\nu\right)}\right)_{L_2\left(a,b\right)}$ -- скалярное произведение в $W_2^k [a,b]$, а $\left(\cdot,\cdot\right)_{L_2\left(a,b\right)}$ -- скалярное произведение в $L_2\left(a,b\right)$.

Обозначим через $\mathcal{W}$ замкнутое подпространство $W_2^1\left(\Gamma_m\right)$, состоящее из всех наборов $[y_1,\dots,y_m]$, удовлетворяющих условиями (\ref{15}), (\ref{16}) и $y_1\left(0\right)=0$.

Также введем пространство $W_2^k\left(\widetilde{\Gamma}_m\right)=W_2^k[0,T_1]\oplus\displaystyle\bigoplus_{j=2}^m W_2^k[0,q^{-1}\left(T_j-\left(q-1\right)T_1\right)]$.

\begin{lemma}\label{L1}Если $y\in W_2^1\left(\Gamma_m\right)$ является решением вариационной задачи (\ref{12}), (\ref{14})--(\ref{17}), то 
\begin{equation}
B\left(y, w\right):=\sum_{j=1}^m \alpha_j\displaystyle\int_0^{T_j} \ell_j y\left(t\right)\ell_j  w\left(t\right) dt =0 \quad \forall w\in\mathcal{W}.
\label{18}
\end{equation}
Обратно, если для некоторого $y\in W_2^1\left(\Gamma_m\right)$ выполняется (\ref{12}), (\ref{15}), (\ref{16}) и (\ref{18}), то $y$ является решением задачи (\ref{12}), (\ref{14})--(\ref{17}).
\end{lemma}

\begin{proof}
Пусть $y\in W_2^1\left(\Gamma_m\right)$ -- решение задачи (\ref{12}), (\ref{14})--(\ref{17}). Тогда для произвольной фиксированной функции $ w\in\mathcal{W}$ сумма $y+s w$ принадлежит $W_2^1\left(\Gamma_m\right)$ при любом $s\in\mathbb{R}$ и удовлетворяет условиям  (\ref{12}), (\ref{15}), (\ref{16}). Нетрудно видеть, что $$\mathcal{J}\left(y+s w\right) = \mathcal{J}\left(y\right)+2sB\left(y, w\right)+s^2\mathcal{J}\left( w\right).$$ Так как  $\mathcal{J}\left(y+s w\right)\geqslant\mathcal{J}\left(y\right)$ для всех $s\in\mathbb{R}$, то выполняется (\ref{18}).

Обратно, пусть  $y\in W_2^1\left(\Gamma_m\right)$ удовлетворяет условиям (\ref{12}), (\ref{15}) и (\ref{16}). Тогда (\ref{18}) влечет
$$\mathcal{J}\left(y+ w\right)=\mathcal{J}\left(y\right)+2B\left(y, w\right)+\mathcal{J}\left( w\right)\geqslant\mathcal{J}\left(y\right) \quad \forall  w\in\mathcal{W}.$$
Таким образом, $y$ доставляет минимум функционалу (\ref{17}) при условиях (\ref{12}), (\ref{15}) и (\ref{16}). 
\end{proof}

Преобразуем (\ref{18}), сделав замену переменных в членах, содержащих $w_j\left(q^{-1}\left(t-\left(q-1\right)T_1\right)\right)$ и $w_1\left(q^{-1}t\right)$. 
В результате выражение для $B\left(y, w\right)$ примет вид
\begin{multline*}
B\left(y, w\right) = \displaystyle\sum_{j=1}^m\alpha_j\displaystyle\int_0^{T_j}\ell_jy\left(t\right) w_j^\prime\left(t\right)dt + \displaystyle\sum_{j=1}^m \alpha_j b_j\displaystyle\int_0^{T_j}\ell_jy\left(t\right) w_j\left(t\right)dt+
\\
+ q\left(\alpha_1 c_1\displaystyle\int_0^{q^{-1}T_1}\ell_1y\left(qt\right) w_1\left(t\right)dt + \displaystyle\sum_{j=2}^m\alpha_j c_j \displaystyle\int_{\left(q^{-1}-1\right)T_1}^{q^{-1}\left(T_j-\left(q-1\right)T_1\right)} \ell_j y\left(qt+\left(q-1\right)T_1\right) w_j\left(t\right)dt \right).
\end{multline*}
Применяя (\ref{14}) к $ w=\left[ w_1,\dots, w_m\right]\in\mathcal{W}$, можно представить
\begin{multline*}\displaystyle\int_{\left(q^{-1}-1\right)T_1}^{q^{-1}\left(T_j-\left(q-1\right)T_1\right)} \ell_j y\left(qt+\left(q-1\right)T_1\right) w_j\left(t\right)dt =
\\ 
=\displaystyle\int_{0}^{q^{-1}\left(T_j-\left(q-1\right)T_1\right)} \ell_j y\left(qt+\left(q-1\right)T_1\right) w_j\left(t\right)dt + 
\displaystyle\int_{q^{-1}T_1}^{T_1} \ell_j y\left(qt-T_1\right) w_1\left(t\right)dt, \quad j=\overline{2,m}.\end{multline*}
Тогда перепишем (\ref{18}) в эквивалентном виде
\begin{multline}
B\left(y, w\right)=\displaystyle\sum_{j=1}^m\left(\alpha_j\displaystyle\int_0^{T_j}\ell_j y\left(t\right) w_j^\prime\left(t\right)dt+\right.
\\
+\left.\displaystyle\int_0^{T_j}\left(\alpha_jb_j\ell_jy\left(t\right)+\widetilde{\ell_j}y\left(t\right)\right) w_j\left(t\right)dt\right)=0\quad\forall w\in\mathcal{W},
\label{19}
\end{multline}
где
\begin{equation}
\begin{aligned}
\tilde{\ell}_1 y\left(t\right) &= \left\{\begin{array}{cc} q \alpha_1c_1\ell_1 y\left(qt\right), & 0<t<q^{-1}T_1,\\[3mm] q \displaystyle\sum_{k=2}^m \alpha_k c_k\ell_k y\left(qt-T_1\right), & q^{-1}T_1< t<T_1, \end{array}\right. 
\\
\tilde{\ell}_j y\left(t\right) &= \left\{\begin{array}{cc} q \alpha_j c_j\ell_j y\left(qt+\left(q-1\right)T_1\right), & 0<t<q^{-1}\left(T_j-\left(q-1\right)T_1\right),\\[3mm]\displaystyle 0, & q^{-1}\left(T_j-\left(q-1\right)T_1\right)< t<T_j, \end{array}\right. \quad j=\overline{2,m}.
\end{aligned}
\label{20}
\end{equation}

Обозначим через $\mathcal{B}$ краевую задачу для функционально-дифференциальных уравнений второго порядка
\begin{equation}
\mathcal{L}_j y\left(t\right):=-\alpha_j\left(\ell_jy\right)^{\prime}\left(t\right)+\alpha_j b_j \ell_j y\left(t\right) + \widetilde{\ell_j}y\left(t\right)=0,\quad 0<t<l_j,\quad j=\overline{1,m},
\label{21}
\end{equation}
при условиях (\ref{12}), (\ref{14})--(\ref{16}) и условии типа Кирхгофа
\begin{equation}
\alpha_1 y_1^\prime\left(T_1\right)-\displaystyle\sum_{j=2}^m \alpha_jy_j^\prime\left(0\right)+\beta y_1\left(T_1\right)+\gamma y_1\left(q^{-1}T_1\right)=0,
\label{22}
\end{equation}
где $\widetilde{\ell_j}y\left(t\right)$ определены в (\ref{20}), и
\begin{equation}
l_1:=T_1, \quad l_j:=q^{-1}\left(T_j-\left(q-1\right)T_1\right), \quad j=\overline{2,m},
\label{23}
\end{equation}
\begin{equation}
\beta:=\alpha_1 b_1-\displaystyle\sum_{j=2}^m \alpha_j b_j, \quad \gamma:=\alpha_1c_1-\displaystyle\sum_{j=2}^m \alpha_j c_j.
\label{24}
\end{equation}

Имеет место следующее утверждение.

\begin{lemma}\label{L2}Если $y\in W_2^1\left(\Gamma_m\right)$ удовлетворяет условиям (\ref{12}), (\ref{15}), (\ref{16}) и (\ref{18}), то $y\in W_2^2\left(\widetilde{\Gamma}_m\right)$ и является решением краевой задачи $\mathcal{B}$. Обратно, любое решение задачи $\mathcal{B}$ подчиняется условию (\ref{18}).\end{lemma}

\begin{proof}Пусть $y\in W_2^1\left(\Gamma_m\right)$  и удовлетворяет условиям (\ref{12}), (\ref{15}), (\ref{16}) и (\ref{18}). Учитывая, что (\ref{18}) эквивалентно (\ref{19}) и применяя лемму 2 из \cite{3} к (\ref{19}) вместе с (\ref{23}), получаем, что $\ell_j y\left(t\right)\in W_2^1\left[0,l_j\right]$ и 
\begin{equation}
\alpha_1\ell_1 y\left(T_1\right)=\displaystyle\sum_{j=2}^m \alpha_j\ell_j y\left(0\right).
\label{25}
\end{equation}
Тогда, используя (\ref{11}), (\ref{13}) и (\ref{23}), получаем $y_j^\prime \left(t\right)\in W_2^1\left[0,l_j\right]$ для $j=\overline{1,m}$. Следовательно, $y\in W_2^2\left(\widetilde{\Gamma}_m\right)$. 

С учетом (\ref{11}) и (\ref{13}) перепишем (\ref{25}) в эквивалентном виде
\begin{multline}
\alpha_1\left(y_1^\prime\left(T_1\right)+b_1y_1\left(T_1\right)+c_1y_1\left(q^{-1}T_1\right)\right)=
\\
=\displaystyle\sum_{j=2}^m \alpha_j\left(y_j^\prime\left(0\right) +b_jy_j\left(0\right)+c_jy_j\left(\left(q^{-1}-1\right)T_1\right)\right),
\label{26}
\end{multline}
где, согласно (\ref{14}), пределы 
$$y_j\left(\left(q^{-1}-1\right)T_1\right):=\displaystyle\lim_{t \to \left(q^{-1}-1\right)T_1+0} y_j\left(t\right) = \displaystyle\lim_{t \to \left(q^{-1}-1\right)T_1+0} y_1\left(t+T_1\right)=y_1(q^{-1}T_1),$$ 
очевидно, существуют. Таким образом, соотношение (\ref{26}) вместе с (\ref{14}), (\ref{15}) дает (\ref{22}) с (\ref{24}). Наконец, интегрируя по частям в (\ref{19}) и используя (\ref{15}), (\ref{16}) и (\ref{21}), будем иметь 
\begin{equation}
B\left(y, w\right)=  w_1\left(T_1\right)\left(\alpha_1\ell_1 y\left(T_1\right)-\displaystyle\sum_{j=2}^m \alpha_j\ell_j y\left(0\right)\right)+ \displaystyle\sum_{j=1}^m\displaystyle\int_0^{l_j}\mathcal{L}_j y\left(t\right) w_j\left(t\right)dt=0.
\label{27}
\end{equation}
В силу (\ref{25}) и произвольности $ w_j$, из (\ref{27}) получаем (\ref{21}).

Обратно, пусть $y$ -- решение задачи $\mathcal{B}$. Так как (\ref{22}) эквивалентно  (\ref{25}), то второе равенство в  (\ref{27}) выполняется. Интегрируя по частям в  (\ref{27}) приходим к  (\ref{19}). Так как  (\ref{19}) эквивалентно  (\ref{18}), то  (\ref{18}) выполняется.
\end{proof}

Объединив леммы \ref{L1} и \ref{L2}, получаем следующий результат.

\begin{theorem}\label{Th1}Функция $y\in W_2^1\left(\Gamma_m\right)$ является решением вариационной задачи (\ref{12}), (\ref{14})--(\ref{17}) тогда и только тогда, когда $y\in W_2^2\left(\widetilde{\Gamma}_m\right)$ и является решением краевой задачи~$\mathcal{B}$.\end{theorem}

\section{Однозначная разрешимость}  

Установим однозначную разрешимость краевой задачи $\mathcal{B}$, а значит, согласно теореме \ref{Th1}, и вариационной задачи (\ref{12}), (\ref{14})--(\ref{17}).

Сначала докажем два вспомогательных утверждения.

\begin{lemma}\label{L3} Существует $C_1$ такое, что 
\begin{equation}
\|\ell w\|^2_{L_2\left(\Gamma_m\right)}\leqslant C_1 \| w\|^2_{W^1_2\left(\widetilde{\Gamma}_m\right)} \quad\forall w\in\mathcal{W}.
\label{28}
\end{equation}
\end{lemma}

\begin{proof}
Используя (\ref{11}), (\ref{13}) и неравенство
\begin{equation}
\left(a_1+\dots+a_n\right)^2 \leqslant n\left(a_1^2+\dots+a_n^2\right),\quad  a_1,\dots,a_n\in{\mathbb R},
\label{29}
\end{equation}
для $n=3$ получаем
\begin{multline}
\|\ell w\|^2_{L_2\left(\Gamma_m\right)} \leqslant 3 \displaystyle\sum_{j=1}^m \displaystyle\int_{0}^{T_j}\left(  w_j^\prime\left(t\right) \right)^2 dt  + 3 \displaystyle\sum_{j=1}^m b_j^2 \displaystyle\int_{0}^{T_j}  w_j^2\left(t\right)  dt + 
\\
+3\left(c_1^2 \displaystyle\int_{0}^{T_j}  w_1^2\left(q^{-1}t\right)dt + \displaystyle\sum_{j=2}^m c_j^2 \displaystyle\int_{0}^{T_j}  w_j^2\left(q^{-1}\left(t-\left(q-1\right)T_1\right)\right)  dt \right).
\label{30}
\end{multline}
Применяя (\ref{14}) к $w$, имеем
\begin{equation}
\begin{aligned}
\displaystyle\int_{0}^{T_j} w_1^2\left(q^{-1}t\right)dt = q\| w_1\|^2_{L_2\left(0,q^{-1}T_1\right)},& 
\\ 
\displaystyle\int_{0}^{T_j} w_j^2\left(q^{-1}\left(t-\left(q-1\right)T_1\right)\right)dt  = q&\| w_1\|^2_{L_2\left(q^{-1}T_1,T_1\right)} + 
\\
 &+q\| w_j\|^2_{L_2\left(0,q^{-1}\left(T_j-\left(q-1\right)T_1\right)\right)}, \quad j=\overline{2,m}.
\end{aligned}\label{31}
\end{equation}
Подставляя (\ref{31}) в (\ref{30}) и применяя (\ref{16}), мы приходим к (\ref{28}), где константа $C_1$ не зависит от $ w$.
\end{proof}

\begin{lemma}\label{L4}Существует $C_2>0$ такое, что 
$$\mathcal{J}\left( w\right)\geqslant C_2\| w\|^2_{W^1_2\left(\widetilde{\Gamma}_m\right)} \quad\forall w\in\mathcal{W}.$$
\end{lemma}

\begin{proof}
Предположим противное. Пусть существует $w_{\left(n\right)}\in\mathcal{W}$, $n\in\mathbb{N}$, такое что
$$\mathcal{J}\left( w_{\left(n\right)}\right)\leqslant\frac{1}{n}\| w_{\left(n\right)}\|^2_{W^1_2\left(\widetilde{\Gamma}_m\right)}.$$
Без ограничения общности можно считать, что 
\begin{equation}\| w_{\left(n\right)}\|_{W^1_2\left(\widetilde{\Gamma}_m\right)}=1. \label{-10}\end{equation}
Тогда
\begin{equation}\mathcal{J}\left( w_{\left(n\right)}\right)\leqslant\frac{1}{n}, \quad n\in\mathbb{N}.\label{32}\end{equation}

Используя (\ref{11}), (\ref{13}) и неравенство (\ref{29}) для $n=3$, получаем
\begin{equation}
\begin{aligned}
&\left( w_1^\prime\left(t\right)\right)^2\leqslant 3\left(\left(\ell_1 w\left(t\right)\right)^2+b_1^2 w_1^2\left(t\right)+c_1^2 w_1^2\left(q^{-1}t\right) \right), 
\\ 
&\left( w_j^\prime\left(t\right)\right)^2\leqslant 3\left(\left(\ell_j w\left(t\right)\right)^2+b_j^2 w_j^2\left(t\right)+c_j^2 w_j^2\left(q^{-1}\left(t-\left(q-1\right)T_1\right)\right) \right), \quad j=\overline{2,m}.
\end{aligned}\label{33}
\end{equation}
Проинтегрировав (\ref{33}) от $0$ до $T_j$ и умножив на $\alpha_j$, а затем просуммировав по $j=\overline{1,m}$, приходим к оценке
\begin{multline}
\frac{1}{3}\displaystyle\sum_{j=1}^m\alpha_j\displaystyle\int_{0}^{T_j}\left(  w_j^\prime\left(t\right) \right)^2 dt\leqslant \displaystyle\sum_{j=1}^m\alpha_j\displaystyle\int_{0}^{T_j}\left( \ell_j w\left(t\right) \right)^2 dt+\displaystyle\sum_{j=1}^m\alpha_j b_j^2\displaystyle\int_{0}^{T_j} w_j^2\left(t\right)dt+
\\
+\alpha_1 c_1^2\displaystyle\int_{0}^{T_1} w_1^2\left(q^{-1}t\right)dt + \displaystyle\sum_{j=2}^m \alpha_j c_j^2\displaystyle\int_{0}^{T_j} w_j^2\left(q^{-1}\left(t-\left(q-1\right)T_1\right)\right)dt.
\label{34}
\end{multline}
Согласно лемме 5 в \cite{3} имеем  
\begin{equation}\| w^\prime\|^2_{L_2\left(\Gamma_m\right)}\geqslant C_3\| w\|^2_{W^1_2\left(\widetilde{\Gamma}_m\right)} \quad\forall w\in\mathcal{W},\label{35}\end{equation}
где $C_3>0$, $ w^\prime=\left[ w_1^\prime,\dots, w_m^\prime\right]$. Подставляя (\ref{31}) в (\ref{34}) и используя оценку (\ref{35}), будем иметь 
\begin{equation}\frac{\alpha C_3}{3}\| w\|^2_{W^1_2\left(\widetilde{\Gamma}_m\right)}\leqslant\mathcal{J}\left( w\right)+K\| w\|^2_{L_2\left(\Gamma_m\right)}, \quad\alpha:=\displaystyle\min_{j=\overline{1,m}}\alpha_j >0.\label{36}\end{equation}

Так как $W^1_2\left(\Gamma_m\right)$ компактно вложено  в $L_2\left(\Gamma_m\right)$, то в силу (\ref{-10}) существует подпоследовательность $\{ w_{\left(n_k\right)}\}_{k\in\mathbb{N}}$ фундаментальная в $L_2\left(\Gamma_m\right)$. Тогда из неравенства (\ref{36}) получаем
$$\frac{\alpha C_3}{3}\| w_{\left(n_k\right)}- w_{\left(n_l\right)}\|^2_{W^1_2\left(\widetilde{\Gamma}_m\right)}\leqslant\mathcal{J}\left( w_{\left(n_k\right)}- w_{\left(n_l\right)}\right)+K\| w_{\left(n_k\right)}- w_{\left(n_l\right)}\|^2_{L_2\left(\Gamma_m\right)}.$$
Кроме того, используя (\ref{29}) при $n=2$ и (\ref{32}), имеем
$$\mathcal{J}\left( w_{\left(n_k\right)}- w_{\left(n_l\right)}\right)\leqslant\frac{2}{n_k}+\frac{2}{n_l}.$$
Таким образом, последовательность $\{ w_{\left(n_k\right)}\}_{k\in\mathbb{N}}$ является фундаментальной в $W^1_2\left(\widetilde{\Gamma}_m\right)$ и сходится к некоторой функции $ w_{\left(0\right)}\in\mathcal{W}$.

В силу леммы \ref{L3} сходимость $ w_{\left(n_k\right)}$ к $ w_{\left(0\right)}$ в $W^1_2\left(\widetilde{\Gamma}_m\right)$ влечет сходимость $\ell w_{\left(n_k\right)}$ к $\ell w_{\left(0\right)}$ в $L_2\left(\Gamma_m\right)$. Следовательно, учитывая (\ref{32}), получаем
$$\|\ell w_{\left(0\right)}\|^2_{L_2\left(\Gamma_m\right)}=\displaystyle\lim_{k \to \infty}\|\ell w_{\left(n_k\right)}\|^2_{L_2\left(\Gamma_m\right)}\leqslant \frac{1}{\alpha}\displaystyle\lim_{k \to \infty}\mathcal{J}\left(w_{\left(n_k\right)}\right)=0,$$
то есть $\ell w_{\left(0\right)}=0$. Таким образом, $ w_{\left(0\right)}$ является решением задачи Коши (\ref{11})--(\ref{15}) при $y_0=0$ и $u_j\left(t\right)\equiv 0$, $j=\overline{1,m}$. В силу единственности решения задачи Коши имеем $\ w_{\left(0\right)}\equiv0$, что противоречит (\ref{-10}).
\end{proof}

Докажем основной результат этого раздела.

\begin{theorem}
Краевая задача $\mathcal{B}$ имеет единственное решение $y\in W^1_2\left(\Gamma_m\right) \cap W_2^2\left(\widetilde{\Gamma}_m\right)$. Кроме того, существует $C$ такое, что 
\begin{equation}\|y\|_{W^1_2\left(\Gamma_m\right)}\leqslant C|y_0|.\label{37}\end{equation}
\end{theorem}

\begin{proof}
Введем функцию $\Phi=\left[\Phi_1,\dots,\Phi_m\right]\in W^1_2\left(\Gamma_m\right)$ такую, что 
\[
\Phi_1\left(t\right) = \begin{cases} 
y_0\left(1-\frac{qt}{T_1}\right), & 0 \leqslant t < q^{-1}T_1, \\
0, & q^{-1}T_1 \leqslant t \leqslant T_1,
\end{cases}
\quad \Phi_j\left(t\right) \equiv 0, \quad j = 2, \dots, m.
\]
В силу леммы \ref{L2}, для того чтобы функция $y\in W^1_2\left(\Gamma_m\right)$, удовлетворяющая условиям (\ref{12}), (\ref{15}), (\ref{16}), была решением краевой задачи $\mathcal{B}$, необходимо и достаточно, чтобы выполнялось условие (\ref{18}). Другими словами, $y$ является решением краевой задачи $\mathcal{B}$ тогда и только тогда, когда $x:=y-\Phi\in\mathcal{W}$ (что эквивалентно условиям (\ref{12}), (\ref{15}), (\ref{16})) и
\begin{equation}B\left(\Phi, w\right)+B\left(x, w\right)=0 \quad\forall w\in\mathcal{W}\label{38}\end{equation}
(что, в свою очередь, эквивалентно условию (\ref{18})).

Покажем, что существует единственная функция $x$ такая, что выполняется (\ref{38}). Так как  $B\left( w, w\right)=\mathcal{J}\left( w\right)$, то в силу лемм \ref{L3} и \ref{L4} билинейная форма $B\left( \cdot,\cdot \right)$ задает на $\mathcal{W}$ эквивалентное скалярное произведение. Кроме того, используя (\ref{13}), (\ref{14}), (\ref{18}) и неравенство Коши--Буняковского, получаем оценку
\begin{equation}\left|B\left(\Phi, w\right)\right|= \alpha_1\left|\displaystyle\int_{0}^{T_1}\ell_1\Phi\left(t\right)\ell_1 w\left(t\right) dt \right| \leqslant M|y_0| \| w\|_{\mathcal{W}},\label{39}\end{equation}
где $\| w\|_{\mathcal{W}}=\sqrt{\left( w, w\right)_{\mathcal{W}}}$. Таким образом, по теореме Рисса об общем виде линейного непрерывного функционала в гильбертовом пространстве существует единственная функция $x\in\mathcal{W}$ такая, что выполняется (\ref{38}). Следовательно краевая задача $\mathcal{B}$ имеет единственное решение $y=\Phi+x$.

Наконец, из  (\ref{38}) и (\ref{39}) вытекает
\begin{equation}\|x\|_{\mathcal{W}}\leqslant M|y_0|.\label{40}\end{equation}
Тогда, используя (\ref{40}) и выражение 
$$\|\Phi\|^2_{W^1_2\left(\Gamma_m\right)} = \|\Phi_1\|^2_{W^1_2\left[0,q^{-1}T_1\right]} = \frac{T_1^2+3q^2}{3qT_1}y_0^2,$$
получаем оценку (\ref{37}).

\end{proof}


\end{document}